\documentclass{article}

\usepackage{graphicx}
\usepackage{placeins}
\usepackage{amsmath, amsfonts, amssymb, amsbsy, amsthm} 

\theoremstyle{plain}
\newtheorem{theorem}{Theorem}[section]
\newtheorem{corollary}[theorem]{Corollary}
\newtheorem{lemma}[theorem]{Lemma}

\theoremstyle{definition}

\title{Second-Order Cone Programming for P-Spline Simulation Metamodeling}

\begin{document}

\author{
Yu Xia \thanks{
Lakehead University,  Canada 
\texttt{yxia@lakeheadu.ca}
Research supported in part by NSERC.
}
\and
Farid Alizadeh\thanks{
Management Science and Information Systems,
Rutgers University, NJ,  USA
\texttt{farid.alizadeh@rutgers.edu}
}
}

\date{April 2015}

\maketitle

\section*{Abstract}
This paper approximates simulation models by B-splines with a penalty on high-order finite differences of the coefficients of adjacent B-splines.  The penalty prevents overfitting.  The simulation output is assumed to be nonnegative.   The nonnegative spline simulation metamodel is casted as a second-order cone programming model, which can be solved efficiently by modern optimization techniques.  The method is implemented in MATLAB/GNU Octave.

\section{Introduction}
People use computer simulation to study complex systems that prohibit analytical evaluations, in order to have a basic understanding of the system, or to find robust decisions or policies, or to compare different decisions or policies~\cite{KSLC2005}. 
Simulation is applied in various areas~\cite{Banks2000, Law1991, Oden2006}, and it is considered as one of the three most important operations research techniques~\cite{LMH}.
Let $y$ represent the response and $x$ represent the input of a system.  A simulation model can then be written as 
\[ y = f(x) . \] 
   In situations where the systems are so complex that even their valid simulation models can't be evaluated in reasonable time, metamodels, or models of models~\cite{Kleijnen75}, are constructed to approximate the simulation models.
Advantages of the simulation metamodel include ``model simplification, enhanced exploration and interpretation of the model, generalization to other models of the same type, sensitivity analysis, optimization, answering inverse questions, and providing the researcher with a better understanding of the behaviour of the system under study and the interrelationships among the variables'' \cite{FLP88}.
      
Parametric polynomial response surface approximation is the most popular technique for building metamodels \cite{Barton98}.  To determine and quantify the unknown or too complex relationship between the response variables and the experimental factors assumed to influence the response, in response surface methodolgy -- introduced by Box and Wilson~\cite{BoxWilson51}, a mathematical model is constructed to fit the data collected from a series of experiments, and the optimal settings of the experimental factors is determined~\cite{MR861118, MR1447628, MR2464113}. Usually the mathematical model is a first or second order polynomial, called a response surface.

By Weierstrass approximation theorem, every continuous function can be uniformly approximated as closely as desired by a polynomial.  Polynomials are easy to compute and have continuous derivatives of all orders.  On the other hand, polynomials are inflexible: their values on a complex plane are determined by an arbitrarily small set \cite[Theorem 3.23]{MR2348176}; they oscillate increasingly with the increase in the order of the polynomials, while high-order is required for suitable accuracy in polynomial approximation; the Runge phenomenon~\cite{Runge1901} is the classic example of divergent polynomial interpolation.  A polynomial fits data nicely near one data point may display repulsive features at parts of the curve not close to that particular data point.

Approximation by splines (smooth piecewise polynomials) overcomes the inflexibility of polynomial approximation.  In practice, B-Splines, which are first thoroughly studied by~\cite{MR0016705}, are widely used in approximation, as there are good properties associated with B-splines~\cite{deBoor01}.  Especiallarly, compared with representations by splines in truncated power basis---defined as $\big\{(x-t_i)_+^j/j! \colon  (j=1, \dots k) \big\}$ for node $t_i$, B-spline representations are relatively well-conditioned and involve fewer basis functions computationally.
Let $\mathbf{t} \equiv (t_i)$ be a nondecreasing sequence.  The $i$th (normalized) B-spline basis function of order $1$ for the knot sequence $\mathbf{t}$ is defined as follows:
\[ B_{i1\mathbf{t}}(x) \equiv \chi_i (x) \equiv \begin{cases} 1  & t_i \leq x < t_{i+1} \\ 0  & \text{otherwise} \end{cases} . \]
When it can be inferred from the context, the knots $\mathbf{t}$ and variable $x$ are omitted in the notations for B-spline representations.  Denote 
\[ \omega_{ik}(x ) \equiv \begin{cases} \frac{x - t_i}{t_{i+k-1} - t_i}, & t_i \neq t_{i+k-1} \\ 0 & \text{otherwise} \end{cases} . \]
 For $k >1$, the $i$th B-Spline basis function of order $k$ for knot sequence $\mathbf{t}$ can be obtained recursively by:
\begin{equation} \label{BSplineDef} B_{ik} \equiv \omega_{ik} B_{i, k-1} + (1-\omega_{i+1, k}) B_{i+1, k-1} . \end{equation}

The spline basis function $B_{ik\mathbf{t}}$ depends only on the knots $t_i, \dots, t_{i+k}$.   It is positive on the interval $(t_i, t_{i+k})$ and is zero elsewhere.   The Curry-Schoenberg Theorem~\cite{MR0218800} describes how to construct B-spline basis for the linear space of piecewise polynomial functions satisfying predefined continuity conditions based on the muliplicity of the knots.  This property simplifies the approximation of functions with required degree of smoothness, compared with truncated power spline approximation, where additional constraints on smoothness need to be included in the model.
The de Boor algorithm~\cite{MR0338617} is a well-conditioned yet efficient way of evaluating B-splines.

\paragraph{P-Spline approximation.}
To fit the metamodel with data collected from experiments, i.e., to find the parameters for the B-spline approximation, we conside P-spline regression \cite{MR1435485}.   The objective function of a P-spline regression combines B-splines with a penalty on high-order finite differences of the coefficients of adjacent B-splines.  Similar to the smoothing term in the loss function for  smoothing spline regression~\cite{MR0295532, MR0167768}, the penalty in P-spline regression loss function prevents overfitting, i.e., the penalty  reduces the variation of the fitted curve caused by data error.  Compared with smoothing splines,  P-splines are relatively inexpensive to compute and without the complexity of choosing the optimal number and positions of knots --- too few data points causes under fitting while too many data points results in overfitting.  An algorithm of determining the number of knots for P-spline regression is given in~\cite{MR1944261}.

Let $(\alpha_i)$ represent the B-spline coefficients.  The second-order differences of the adjacent B-spline coefficients for the knot sequence $\mathbf{t}$ are
\[ \left( \alpha_j - \alpha_{j-1} \right) - \left( \alpha_{j-1} - \alpha_{j-2} \right) =  \alpha_j - 2\alpha_{j-1} + \alpha_{j-2}.
\]

Denote the parameter controlling the smoothness of the fit by $\lambda$.
The least squares objective function (loss function) of the regression of $m$ data points $(x_i, y_i)$ using $n$ B-spline basis functions of order four with a penalty on second-order differences of the B-spline coefficients, i.e. the P-spline regression loss function for fitting the metamodel studied in this paper, is
\begin{equation} \label{eq:PSpline}
\min_\alpha \quad \sum_{i=1}^m \left[ y_i - \sum_{j=1}^n \alpha_j B_{j4}(x_i) \right]^2 + \lambda \sum_{j=3}^n \left( \alpha_j - 2\alpha_{j-1} + \alpha_{j-2} \right)^2.
\end{equation}

\paragraph{Nonnegative model fitting.}  In many applications, the response of the system is known or required to be nonnegative or above some threshold; for instance, when the output of the system describes duration, productions, prices, demand, sales, wages, amount of precipitation, probability mass, etc; see~\cite{MR2402223, Ramsay88}).
Because of the noisy or tendency in the data, quite often, the fitted curve doesn't exhibit nonnegativitiy, even though it should be.  For instance, let $(x_i, y_i)$ be the monthly precipitation amount at some region, where $x_i$ is the variable for months and $y_i$ is the rain fall amount.   The rain fall amount may be decreasing during some period till at some months there is little or no rain, then it may increase again.   Because of the increase and decrease trend before and after these certain months, the fitted curve may have negative values at these points for $y$~\cite{MR2168993}.
Even if the data points are nonnegative, without imposing the nonnegativity constraints, the resulting models may take negative values at some areas~\cite{XiaMcNicholas}.
In~\cite{ MR2402223} a numerical example of cubic spline approximation of arrival-rate for an e-mail data set shows that the maximum likelihood spline takes negative values in a significant time period with all positive data points, and the estimation problem may even be unbounded and thus ill-posed.

To obtain a satisfiable and sometimes meaningful model, the nonnegativity constraint on the output needs to be imposed on the regression.  The nonnegative cubic spline approximation in truncated power basis is considered in \cite{MR2402223}.   Constrained smoothing spline approximation is studied in \cite{MR883352}, but they acknowledge the computational difficulty in their approach.
Since the B-spline basis functions are nonnegative, imposing positivity on B-spline coefficients \cite{MR1631345} or integrating B-splines with positive coefficients (I-spline \cite{Ramsay88}) preserves positivity in regression.  But this approach excludes some classes of positive splines and thus reduces the accuracy of the regression.
It is proved in~\cite{MR0378357} that errors in approximation of nonnegative functions by B-splines of order $k$ (degree $< k$) with nonnegative coefficients are bigger in magnitude compared with errors in approximation by nonnegative splines of the same order, and the difference between the magnitudes of the errors increases with the order of the splines.
The two approximation schemes give errors of the same magnitude only if $k \leq 2$, i.e. approximation by piece wise constant or piece wise linear functions.
Because of the approximation and computational advantage of P-splines, this paper focuses on nonnegative P-spline approximation.

To simplify notation, in this paper, we concatenate vectors row wise by `,' and concatenate vectors column wise by ';'; for instance, the adjoining of vectors $x$, $y$, and $z$ can be represented as
\[ \begin{pmatrix} x \\ y \\ z \end{pmatrix} = \left(x^\top, y^\top, z^\top \right) = \left( x; y; z \right) .   \]

\section{Nonnegative Cubic Polynomials}
In the context below, matrices are represented by capital letters: $A\equiv [a_{ij}]$, where the element of matrix $A$ at both of its $i$th row and $j$th column is denoted as $a_{ij}$.  Let $A \succeq 0$ represent the symmetric matrix $A$ being positive semidefinite.  For two matrices $A$ and $B$ of the same size, let $A \bullet B$ denote their Hadamard product:
\[ A \bullet B = \sum_{ij} a_{ij} b_{ij} . \]

By Markov-Lukacs theorem \cite{Nes00FSS}, a cubic polynomial $p(x) \equiv \beta_3 x^3 + \beta_2 x^2 + \beta_1 x + \beta_0$ is nonnegative on the interval $[t_i, t_{i+1}]$ if and only if there exist $c_1, c_2, d_1, d_2 \in \mathbb{R}$ such that $p(x)$ can be represented as
\[ p(x) = (x - t_i) (c_1 x + c_2)^2 + (t_{i+1} - x)(d_1 x + d_2)^2 .   \]
Denote 
\[ H \equiv \begin{bmatrix} x^2 & x \\ x & 1 \end{bmatrix} .\]
Then by \cite[Theorem 1]{Nes00FSS}, the above representation is equivalent to: existing $C \equiv [c_{ij}] \succeq 0$, $D \equiv [d_{ij}] \succeq 0$ such that
\[ (c_1 x + c_2)^2 = C \bullet H, \quad  (d_1 x + d_2)^2 = D \bullet H .\] 
Because $C \succeq 0$ is equivalent to: $c_{11} \geq 0, c_{22} \geq 0, c_{12}^2 \leq c_{11}c_{22}$,
$p(x)$ is nonnegative on $[t_i, t_{i+1}]$ if and only if there exist $c_{11}, c_{12}, c_{22}, d_{11}, d_{12}, d_{22} \in \mathbb{R}$ such that
 
\begin{equation} \label{eq:pcubic}
\begin{gathered}
\beta_3 = c_{11} -  d_{11}, \\
\beta_2 =  -t_i c_{11} + 2 c_{12} + t_{i+1} d_{11} - 2 d_{12}, \\
\beta_1 = - 2 t_i c_{12} + c_{22} + 2 t_{i+1}d_{12} - d_{22}, \\
\beta_0= -t_i c_{22} + t_{i+1} d_{22}\\
c_{11}, c_{22}, d_{11}, d_{22} \geq 0 \\
c_{12}^2 \leq c_{11}c_{22}, \quad d_{12}^2 \leq d_{11}d_{22}.
\end{gathered}
\end{equation}
 
\section{Nonnegative Representations By B-Splines  Of Order Four}
Based on the definition of B-spline basis functions~\eqref{BSplineDef}, the $i$th B-spline basis function of order three for knot sequence $\mathbf{t}$  is
\[ B_{i3} = \omega_{i3}\omega_{i2}\chi_i + \left[ \omega_{i3}\left(1 - \omega_{i+1, 2}\right) + \left(1 - \omega_{i+1,3}\right) \omega_{i+1,2} \right]\chi_{i+1} + \left(1 - \omega_{i+1,3}\right)\left(1 - \omega_{i+2,2}\right)\chi_{i+2} . \]
And the $i$th B-spline basis function of order four for knot sequence $\mathbf{t}$ is
\[ \begin{split}
 B_{i4} & = \omega_{i4} B_{i,3} + (1 - \omega_{i+1, 4}) B_{i+1, 3} \\
 &= \omega_{i4}\omega_{i3}\omega_{i2} \chi_{i} + \left[\omega_{i4}\left(\omega_{i3}\left(1 - \omega_{i+1, 2}\right) + \left(1 - \omega_{i+1,3}\right) \omega_{i+1,2} \right) + \left(1 - \omega_{i+1,4}\right)   \omega_{i+1,3}\omega_{i+1,2} \right]\chi_{i+1} \\
 & \quad + \left[ \omega_{i4}  \left(1 - \omega_{i+1,3}\right)\left(1 - \omega_{i+2,2}\right) + \left(1 - \omega_{i+1,4}\right) 
 \left( \omega_{i+1,3}\left(1 - \omega_{i+2, 2}\right) + \left(1 - \omega_{i+2,3}\right) \omega_{i+2,2} \right)  \right] \chi_{i+2} \\
 & \quad + \left(1 - \omega_{i+1,4}\right)\left(1 - \omega_{i+2,3}\right)\left(1 - \omega_{i+3,2}\right)\chi_{i+3}
 \end{split} \]
Hence, on the interval $[t_i, t_{i+1})$, the B-spline $\sum_i \alpha_i B_{i4\mathbf{t}}(x)$ is 
\[
\begin{split}
& \Big\{ \alpha_i  \omega_{i4}\omega_{i3}\omega_{i2} + \alpha_{i-1} \left[\omega_{i-1,4}\left(\omega_{i-1,3}\left(1 - \omega_{i 2}\right) + \left(1 - \omega_{i3}\right) \omega_{i2} \right) + \left(1 - \omega_{i,4}\right)   \omega_{i3}\omega_{i2} \right] \\
&  + \alpha_{i-2}\left[ \omega_{i-2,4}  \left(1 - \omega_{i-1,3}\right)\left(1 - \omega_{i,2}\right) + \left(1 - \omega_{i-1,4}\right) 
 \left( \omega_{i-1,3}\left(1 - \omega_{i, 2}\right) + \left(1 - \omega_{i,3}\right) \omega_{i,2} \right)  \right] \\
 &  + \alpha_{i-3} \left(1 - \omega_{i-2,4}\right)\left(1 - \omega_{i-1,3}\right)\left(1 - \omega_{i2}\right) \Big\} \chi_i (x) \\
 \end{split}
\]
Given a finite knot sequence $\mathbf{t} \equiv (t_1, \dots, t_n)$, define 
\[ t_{1-k} = \dots = t_0 = t_1, \qquad  t_n = t_{n+1} = \dots = t_{n+k} .  \]
For $u=0, 1, 2, 3$ and $v = i-3, i-2, i-1, i$, let $a_{u, v}^{(i)}$ denote the coefficient of $x^u$ associated with $\alpha_v$ of the polynomial on the interval $[t_i, t_{i+1})$. 

\[ \sum_i \alpha_i B_{i4} = \sum_i \sum_{j=i-3}^i \alpha_j \left( \sum_{l=0}^3 a_{lj}^{(i)} x^l \right) \chi_i , \]
where for $j \leq 0$, define $\alpha_j=0$ and $a^{(i)}_{u,j} = 0$.  In other words,

\[
\begin{split}
\sum_{l=0}^3 a_{li}^{(i)}x^l & = \omega_{i4}\omega_{i3}\omega_{i2} \\
\sum_{l=0}^3 a_{l, i-1}^{(i)}x^l &=  \omega_{i-1,4}\left(\omega_{i-1,3}\left(1 - \omega_{i 2}\right) + \left(1 - \omega_{i3}\right) \omega_{i2} \right) + \left(1 - \omega_{i,4}\right)   \omega_{i3}\omega_{i2}  \\
\sum_{l=0}^3 a_{l, i-2}^{(i)}x^l & = \omega_{i-2,4}  \left(1 - \omega_{i-1,3}\right)\left(1 - \omega_{i,2}\right) + \left(1 - \omega_{i-1,4}\right) 
 \left( \omega_{i-1,3}\left(1 - \omega_{i, 2}\right) + \left(1 - \omega_{i,3}\right) \omega_{i,2} \right)  \\
\sum_{l=0}^3 a_{l, i-3}^{(i)}x^l & = \left(1 - \omega_{i-2,4}\right)\left(1 - \omega_{i-1,3}\right)\left(1 - \omega_{i2}\right) .
\end{split}
\]

Let $1/(t_j - t_l) \equiv 0$ for $t_j = t_l$.  Then $a_{u, v}^{(i)}$ can be represented in terms of $\mathbf{t}$ as below:

\[
\begin{split}
b_{i-3} & \equiv \frac{1}{(t_{i+1}-t_{i-2})(t_{i+1}-t_{i-1})(t_{i+1}-t_{i})} \\
a^{(i)}_{3, i-3} & = -b_{i-3}, \quad
 a^{(i)}_{2, i-3}  = 3t_{i+1} b_{i-3}, \quad 
 a^{(i)}_{1, i-3}  = -t_{i+1}a^{(i)}_{2, i-3}, \quad 
 a^{(i)}_{0, i-3}  = t_{i+1}^3 b_{i-3} \\
b_{i-2} & \equiv  \frac{1}{(t_{i+2}-t_{i-1})(t_{i+1}-t_{i-1})(t_{i+1}-t_{i})} , \quad
b_{i-1}  \equiv \frac{1}{(t_{i+2}-t_{i-1})(t_{i+2}-t_{i})(t_{i+1}-t_{i})} \\
a^{(i)}_{3, i-2} & = b_{i-3} + b_{i-2} + b_{i-1} \\
a^{(i)}_{2, i-2} & = -(t_{i-2}+ 2t_{i+1}) b_{i-3}  - (t_{i-1} + t_{i+1} + t_{i+2}) b_{i-2} -(t_i + 2t_{i+2}) b_{i-1} \\
a^{(i)}_{1, i-2} & = (2t_{i-2}t_{i+1} + t_{i+1}^2) b_{i-3} + (t_{i-1}t_{i+1} + t_{i-1}t_{i+2} + t_{i+1}t_{i+2})b_{i-2} + (2t_it_{i+2} + t_{i+2}^2)b_{i-1}\\
a^{(i)}_{0, i-2} & = -t_{i-2}t_{i+1}^2b_{i-3} - t_{i-1}t_{i+1}t_{i+2}b_{i-2} - t_i t_{i+2}^2b_{i-1}\\
b_{i-4} & \equiv \frac{1}{(t_{i+3} - t_i)(t_{i+2}-t_i)(t_{i+1}-t_i)} \\
a^{(i)}_{3, i-1} & = -b_{i-4} - b_{i-2} - b_{i-1}\\
a^{(i)}_{2, i-1} & = (t_{i+3} + 2t_i)b_{i-4} + (t_{i+1} + 2t_{i-1})b_{i-2} + (t_{i+2} + t_i + t_{i-1})b_{i-1} \\
a^{(i)}_{1, i-1} & = -(t_i^2 + 2t_{i+3}t_i)b_{i-4} - (2t_{i+1}t_{i-1} + t_{i-1}^2)b_{i-2}- [t_{i+2}(t_{i-1}+t_i) + t_{i-1}t_i]b_{i-1} \\
a^{(i)}_{0, i-1} & = t_{i+3}t_i^2b_{i-4} + t_{i+1}t_{i-1}^2b_{i-2} + t_{i+2}t_it_{i-1}b_{i-1}\\
a^{(i)}_{3,i} &= b_{i-4}, \quad 
a^{(i)}_{2,i} = -3t_ib_{i-4} , \quad 
a^{(i)}_{1,i} = -t_ia^{(i)}_{2,i} , \quad
a^{(i)}_{0,i}  = -t_i^3 b_{i-4} \\
\end{split}
\]

Denote
\[ \Delta_i \equiv t_{i+1} - t_i . \]
For equally spaced knot sequence, i.e., $ \Delta = \Delta_i = \Delta_{i+1} = \dots$, the above expression for $a_{u, v}^{(i)}$ can be simplified:
\[
\begin{aligned}
a^{(i)}_{3, i-3} & = -\frac{1}{6\Delta^3} & 
a^{(i)}_{2, i-3} & = \frac{t_i + \Delta}{2\Delta^3} & 
a^{(i)}_{1, i-3} & = -\frac{(t_i + \Delta)^2}{2\Delta^3} & 
a^{(i)}_{0, i-3} & = \frac{(t_i + \Delta)^3}{6\Delta^3} \\
a^{(i)}_{3, i-2} & = \frac{1}{2\Delta^3} & 
a^{(i)}_{2, i-2} & = -\frac{3t_i + 2 \Delta}{2\Delta^3} & 
a^{(i)}_{1, i-2} & = \frac{3t_i^2 + 4t_i \Delta}{2\Delta^3} & 
a^{(i)}_{0, i-2} & = \frac{-3t_i^3 - 6 t_i^2 \Delta + 4 \Delta^3}{6\Delta^3} \\
a^{(i)}_{3, i-1} & =  -\frac{1}{2\Delta^3} & 
a^{(i)}_{2, i-1} & =  \frac{3t_i + \Delta}{2\Delta^3} & 
a^{(i)}_{1, i-1} & =  \frac{-3t_i^2 -2t_i \Delta + \Delta^2}{2\Delta^3} & 
a^{(i)}_{0, i-1} & =  \frac{3t_i^3 + 3t_i^2\Delta -3t_i\Delta^2 + \Delta^3}{6\Delta^3}\\ 
a^{(i)}_{3, i} & = \frac{1}{6\Delta^3} & 
a^{(i)}_{2, i} & = -\frac{t_i}{2\Delta^3} & 
a^{(i)}_{1, i} & = \frac{t_i^2}{2\Delta^3} & 
a^{(i)}_{0, i} & = -\frac{t_i^3
}{6\Delta^3}\\
\end{aligned}
\]

By \eqref{eq:pcubic}, the B-spline $\sum_i \alpha_i B_{i4}$ is nonnegative on the interval $[t_i, t_{i+1})$ if{f} there exist 
$c_{11}^{(i)}, c_{22}^{(i)}, d_{11}^{(i)}, d_{22}^{(i)}$, such that 
\begin{equation} \label{constraints}
\begin{gathered}
\sum_{j=i-3}^{i} a^{(i)}_{3j} \alpha_j   = c_{11}^{(i)} -  d_{11}^{(i)},\\ 
\sum_{j=i-3}^{i} a^{(i)}_{2j} \alpha_j =  -t_i c_{11}^{(i)} + 2 c_{12}^{(i)} + t_{i+1} d_{11}^{(i)} - 2 d_{12}^{(i)}, \\ 
\sum_{j=i-3}^{i} a^{(i)}_{1j} \alpha_j = - 2 t_i c_{12}^{(i)} + c_{22}^{(i)} + 2 t_{i+1} d_{12}^{(i)} - d_{22}^{(i)}, \\
\sum_{j=i-3}^{i} a^{(i)}_{0j} \alpha_j = -t_i c_{22}^{(i)} + t_{i+1} d_{22}^{(i)} \\
c_{11}^{(i)}, c_{22}^{(i)}, d_{11}^{(i)}, d_{22}^{(i)} \geq 0 \\
\left(c_{12}^{(i)}\right)^2 \leq c_{11}^{(i)} c_{22}^{(i)}, \quad \left(d_{12}^{(i)}\right)^2 \leq d_{11}^{(i)}d_{22}^{(i)}.
\end{gathered}
\end{equation}

\paragraph{The model.} Adding constraints~\eqref{constraints} to the P-spline regression loss function~\eqref{eq:PSpline}, we obtain the formula for fitting the metamodel: 

\begin{equation} \label{prob:full}
\begin{split}
\min_{\alpha, c, d} & \quad \sum_{i=1}^m \left[ y_i - \sum_{j=1}^n \alpha_j B_{j4}(x_i) \right]^2 + \lambda \sum_{j=3}^n \left( \alpha_j - 2\alpha_{j-1} + \alpha_{j-2} \right)^2. \\
\text{s.t.} \quad &
\sum_{j=i-3}^{i} a^{(i)}_{3j} \alpha_j   = c_{11}^{(i)} -  d_{11}^{(i)},\\ 
& \sum_{j=i-3}^{i} a^{(i)}_{2j} \alpha_j =  -t_i c_{11}^{(i)} + 2 c_{12}^{(i)} + t_{i+1} d_{11}^{(i)} - 2 d_{12}^{(i)}, \\ 
& \sum_{j=i-3}^{i} a^{(i)}_{1j} \alpha_j = - 2 t_i c_{12}^{(i)} + c_{22}^{(i)} + 2 t_{i+1} d_{12}^{(i)} - d_{22}^{(i)}, \\
& \sum_{j=i-3}^{i} a^{(i)}_{0j} \alpha_j = -t_i c_{22}^{(i)} + t_{i+1} d_{22}^{(i)} \\
& c_{11}^{(i)}, c_{22}^{(i)}, d_{11}^{(i)}, d_{22}^{(i)} \geq 0 \\
&  \left( c_{12}^{(i)}\right)^2 \leq c_{11}^{(i)} c_{22}^{(i)}, \quad \left( d_{12}^{(i)}\right)^2 \leq d_{11}^{(i)}d_{22}^{(i)} \\
& (i = 1, \dots, n) .
\end{split}
\end{equation}

\paragraph{Variable reduction.}
Let $\alpha$ denote the column vector containing all $\alpha_i$:
$ \alpha \equiv \left( \alpha_1, \alpha_2, \dots, \alpha_n \right)^\top$.
Denote
\[ c_i \equiv \left( c_{11}^{(i)}, c_{22}^{(i)},  c_{12}^{(i)} \right)^\top, \qquad d_i \equiv \left( d_{11}^{(i)}, d_{22}^{(i)},  d_{12}^{(i)} \right)^\top . \]
Let $c$ and $d$ denote the column vectors containing all $c_{lj}^{(i)}$'s and  $d_{lj}^{(i)}$'s:
\[
c \equiv \left( c_1; c_2; \dotsc; c_n \right),
\qquad
d \equiv \left( d_1; d_2; \dotsc; d_n \right).
\]

Constraints~\eqref{constraints} contain $4n$ homongenous equations and $7n$ variables: $\alpha$, $c$, and $d$.
By Curry-Schoenberg theorem~\cite{MR0218800, deBoor01}, the sequence $B_{1, k, \mathbf{t}}$, \dots, $B_{n, k, \mathbf{t}}$ is a basis for the linear space of piecewise polynomials of order $k$ with break sequence $\mathbf{t}$ that satisfies continuity condition specified by the multiplicities of the elements of $\mathbf{t}$.
Since the sub-matrix corresponding to $\alpha$ in the coefficient matrix of the constraints~\eqref{constraints} is the linear transformation from the B-spline basis to the truncated power function basis, the matrix corresponding to $\alpha$ has full column rank.

\begin{lemma}
The coefficient matrix of the equalities in constraints~\textup{\eqref{constraints}} has rank at least $4n - \kappa + \varrho$, where $\kappa$ is the number of total multiple knots -- counted with multiplicities, and $\varrho$ is the number of different multiple knots -- counted without multiplicities.
\end{lemma}
\begin{proof}
The submatrix of the coefficent matrix for the equalities in constraints~\textup{\eqref{constraints}} with columns corresponding to $c$ and $d$ is block diagonal, where the $i$th block is:
\[ G_i \equiv  \begin{bmatrix}
-1  &      &     &  1 \\
t_i & -2   &     & -t_{i+1} & 2   & \\
    & 2t_i & -1  &          & -2t_{i+1} & 1   \\
    &      & t_i &          &           & -t_{i+1}
\end{bmatrix} . \]
The block has rank $4$ if $t_i \neq t_{i+1}$, and it has rank $3$ if $t_i = t_{i+1}$.  Since the coefficient matrix of the equalities in constraints~\eqref{constraints} has $4n$ rows, the statement of the lemma follows.
\end{proof}

\begin{corollary}
If all the knots $(t_i)_{i=1}^{n+1}$ are distinct, then the coefficient matrix of the constraints~\textup{\eqref{constraints}} has full row rank.
\end{corollary}

Therefore, given a distinct knot sequence $\mathbf{t}$, we can use Gauss elimination to represent $\alpha$ by $c$, $d$ and $\mathbf{t}$ in the constraints~\eqref{constraints}.
Since each $\alpha_i$ relates to only $t_i$, $t_{i+1}$, $t_{i+2}$, $t_{i+3}$ in constraints~\eqref{constraints}, we can represent each $\alpha_i$ by at most
variables $c^{(i)}_{11}$, $c^{(i)}_{12}$, $c^{(i)}_{22}$, $d^{(i)}_{11}$, $d^{(i)}_{12}$, $d^{(i)}_{22}$.

For equally spaced knot sequences, below are
representations of $\alpha_{i-3}, \alpha_{i-2}, \alpha_{i-1}, \alpha_i$:

For $4 \leq i \leq n$, omitting the subscript of $t_i$ for simplicity, we have
\begin{equation}
\begin{split}
\alpha_i /\Delta^3& = \left( 2 \frac{t^2}{\Delta^2} + \frac{22t}{3\Delta} + 6 \right) c^{(i)}_{11} + \left( \frac{8t}{3\Delta^2} + \frac{22}{3\Delta} \right)c^{(i)}_{12} + \left( \frac{2t}{3\Delta^3} + \frac{8}{3\Delta^2} \right) c^{(i)}_{22} \\
& - \left( \frac{t^2}{\Delta^2} + \frac{10t}{3\Delta} + \frac{7}{3} \right) d^{(i)}_{11} + \left( \frac{4t^2}{3\Delta^2} + \frac{2t}{3 \Delta^2} - \frac{2}{\Delta} \right)d^{(i)}_{12}  - \left(\frac{2t}{3\Delta^3} + \frac{5}{3\Delta^2} \right) d^{(i)}_{22} \\
\alpha_{i-1}/\Delta^3 & = \left( \frac{t^2}{\Delta^2} + \frac{4t}{3\Delta} \right)c^{(i)}_{11} + \left( \frac{2t}{3\Delta^2} + \frac{4}{3\Delta} \right)c^{(i)}_{12} + \left( \frac{2t}{3\Delta^3} + \frac{5}{3\Delta^2}\right)c^{(i)}_{22}\\
& + \left( \frac{2t}{3\Delta} + \frac{2}{3} \right)d^{(i)}_{11}  + \left( \frac{4t^2}{3 \Delta^3} + \frac{8t}{3 \Delta^2} + \frac{2}{\Delta} \right) d^{(i)}_{12} - \left( \frac{2t}{3 \Delta^3} + \frac{2}{3 \Delta^2} \right) d^{(i)}_{22} \\
\alpha_{i-2} /\Delta^3 & = - \frac{2t}{3\Delta} c^{(i)}_{11}  - \left( \frac{4t}{3\Delta^2} + \frac{2}{3\Delta} \right) c^{(i)}_{12} + \left( \frac{2t}{3\Delta^3} + \frac{2}{3\Delta^2} \right) c^{(i)}_{22} \\
& + \left( \frac{t^2}{\Delta^2} + \frac{2t}{3 \Delta} - \frac{1}{3} \right) d^{(i)}_{11} + \left( \frac{4t^2}{3\Delta^3} + \frac{14t}{3\Delta^3} + \frac{2}{\Delta} \right) d^{(i)}_{12} + \left( - \frac{2t}{3\Delta^3} + \frac{1}{\Delta^2} \right) d^{(i)}_{22} \\
\alpha_{i-3} /\Delta^3 & = \left( - \frac{t^2}{\Delta^2} + \frac{4t}{3\Delta} \right) c^{(i)}_{11} + \left( -\frac{10t}{3\Delta^2} + \frac{4}{3\Delta} \right) c^{(i)}_{12} + \left( \frac{2t}{3\Delta^3} - \frac{1}{3\Delta^2} \right) c^{(i)}_{22} \\
& + \left( 2\frac{t^2}{\Delta^2} - \frac{10t}{3\Delta} + \frac{2}{3} \right)d^{(i)}_{11} + \left( \frac{4t^2}{3\Delta^3} + \frac{20t}{3\Delta^2} - \frac{2}{\Delta} \right) d^{(i)}_{12} + \left( - \frac{2t}{3\Delta^3} + \frac{4}{3\Delta^2} \right) d^{(i)}_{22} 
\end{split}
\end{equation}

For $i=1$:
\[
\begin{split}
\alpha_1 &= 6 \Delta^3 \left[ c_{11}^{(1)} -  d_{11}^{(1)} \right],\\ 
3 t_1 \left[ c_{11}^{(1)} -  d_{11}^{(1)} \right]
&= t_1 c_{11}^{(1)} - 2 c_{12}^{(1)} - t_{2} d_{11}^{(1)} + 2 d_{12}^{(1)}, \\ 
3 t_1^2 \left[ c_{11}^{(1)} -  d_{11}^{(1)} \right]
&= - 2 t_1 c_{12}^{(1)} + c_{22}^{(1)} + 2 (t_{1} + \Delta) d_{12}^{(1)} - d_{22}^{(1)}, \\
 t_1^3 \left[ c_{11}^{(1)} -  d_{11}^{(1)} \right]
&= t_1 c_{22}^{(1)} - (t_{1} + \Delta) d_{22}^{(1)} \\
\end{split}
\]

For $i=2$:
\[
\begin{split}
\alpha_1 &= 4t_2\Delta^2 c_{11}^{(2)} + t \Delta^2 c_{12}^{(2)} + \left( 2 \Delta^3 - 4t_2\Delta^2 \right) d_{11}^{(2)} - 4\Delta^2 d_{12}^{(2)} \\
\alpha_2 &= 6\Delta^2\left(2 t_2 + \Delta \right) c_{11}^{(2)} + 12\Delta^2c_{12}^{(2)} - 12 \Delta^2t_2d_{11}^{(2)} - 12 \Delta^2 d_{12}^{(2)} \\
\alpha_2 &= \frac{6\Delta^2}{\Delta - 2t_2} \Big[ \left( \Delta^2 - 2t_2\Delta - 3t_2^2 \right) c_{11}^{(2)} - 2t_2 c_{12}^{(2)}+ c_{22}^{(2)} + \left(3t_2^2 + 2t_2\Delta - \Delta^2 \right)d_{11}^{(2)} \\
& + 2 \left(t_2 + \Delta\right)d_{12}^{(2)} - d_{22}^{(2)} \Big] \\
\alpha_2 &= \frac{6\Delta^3}{t_2^2 - t_2\Delta^2 + \Delta^3/3}\Big[ \left(t_2^3 + t_2^2 - t_2\Delta^2 + \Delta^3/3 \right) c_{11}^{(2)} - t_2 c_{22}^{(2)} \\
& - \left(t_2^3 + t_2^2 - t_2 \Delta^2 + \Delta^3/2 \right)d_{11}^{(2)} + (t_2 + \Delta) d_{22}^{(2)} \Big]
\end{split}
\]

For $i=3$:
\[
\begin{split}
\alpha_1 & = \Delta \left(t_3^2 - 2 t_3 \Delta \right)c_{11}^{(3)} + \Delta \left(2 t_3 - 2\Delta\right)c_{12}^{(3)}
+ \Delta c_{22}^{(3)} - \Delta \left( \Delta^2 -4t_3\Delta + t_3^2 \right) d_{11}^{(3)}\\
& + \Delta \left(4 \Delta - 2t_3\right) d_{12}^{(3)} - \Delta d_{22}^{(3)} \\
\alpha_1 & = \frac{\Delta^2}{\Delta^2 -t_3\Delta + t_3^2\Delta - t_3^2} \Big[ \left( t_3^3 - 2 t_3^3\Delta + 2t_3^2\Delta - \frac{2t_3\Delta^2}{3} \right)c_{11}^{(3)}  - \Big(\frac{2}{3}\Delta^2 - 2t_3\Delta \\
& + 2t_3^2 \Delta) \Big)c_{12}^{(3)} - t_3 c_{22}^{(3)} - \Big(t_3^3 + t_3^2\Delta^2 - \frac{5}{3}t_3\Delta^2 - 2t_3^3 \Delta + 2 t_3^2\Delta + \frac{\Delta^3}{3} \Big)d_{11}^{(3)} \\
& + \left( \frac{2}{3}\Delta^2 + 2t_3^2 \Delta - 2t_3 \Delta \right) d_{12}^{(3)} + \left(t_3 + \Delta\right) d_{22}^{(3)} \Big] \\
\alpha_2 &= 2 t_3^2\Delta c_{11}^{(3)} + 4t_3\Delta c_{12}^{(3)} + 2\Delta c_{22}^{(3)} + \left(-2t_3^2\Delta + 4 t_3 \Delta^2 \right) d_{11}^{(3)} + \left(4\Delta^2 - 4 t_3\Delta \right) d_{12}^{(3)} - 2\Delta d_{22}^{(3)} \\
\alpha_3 &= \left(3 t_3^2 \Delta + 6 t_3 \Delta^2 \right) c_{11}^{(3)} + \left(6 t_3 \Delta + 6 \Delta^2 \right) c_{12}^{(3)} + 3\Delta c_{22}^{(3)} + \left( 3 \Delta^3 - 3 t_3^2 \Delta \right) d_{11}^{(3)} - 6 t_3 \Delta d_{12}^{(3)} \\
& + 6 \Delta^3 c_{11}^{(3)} - 6 \Delta^3 d_{11}^{(3)} - 3\Delta d_{22}^{(3)} .
\end{split}
\]

Then we can replace $\alpha_i$ in the objective of \eqref{prob:full} by the following relation:

\begin{equation}
\begin{split}
\alpha_1 &= 6 \Delta^3 \left[ c_{11}^{(1)} -  d_{11}^{(1)} \right],\\ 
\alpha_2 &= \frac{6\Delta^3}{t_2^2 - t_2\Delta^2 + \Delta^3/3}\Big[ \left(t_2^3 + t_2^2 - t_2\Delta^2 + \Delta^3/3 \right) c_{11}^{(2)} - t_2 c_{22}^{(2)} \\
& - \left(t_2^3 + t_2^2 - t_2 \Delta^2 + \Delta^3/2 \right)d_{11}^{(2)} + (t_2 + \Delta) d_{22}^{(2)} \Big]\\
\alpha_3 &= \left(3 t_3^2 \Delta + 6 t_3 \Delta^2 \right) c_{11}^{(3)} + \left(6 t_3 \Delta + 6 \Delta^2 \right) c_{12}^{(3)} + 3\Delta c_{22}^{(3)} + \left( 3 \Delta^3 - 3 t_3^2 \Delta \right) d_{11}^{(3)} - 6 t_3 \Delta d_{12}^{(3)} \\
& + 6 \Delta^3 c_{11}^{(3)} - 6 \Delta^3 d_{11}^{(3)} - 3\Delta d_{22}^{(3)}\\ 
i\geq 4: &\\ 
\alpha_i & = \left( 2 t^2\Delta + \frac{22t \Delta^2}{3} + 6 \Delta^3 \right) c^{(i)}_{11} + \left( \frac{8t \Delta }{3} + \frac{22 \Delta^2}{3} \right)c^{(i)}_{12} + \left( \frac{2t}{3} + \frac{8 \Delta}{3} \right) c^{(i)}_{22} \\
& - \left( t^2\Delta + \frac{10t \Delta^2}{3} + \frac{7}{3} \Delta^3 \right) d^{(i)}_{11} + \left( \frac{4t^2 \Delta}{3} + \frac{2t \Delta }{3 } - 2 \Delta^2 \right)d^{(i)}_{12}  - \left(\frac{2t}{3} + \frac{5 \Delta}{3} \right) d^{(i)}_{22}, \\   
\end{split}
\end{equation}

\section{Second-Order Cone Programming}
Index vectors in $\mathbb{R}^n$ from $0$.  A \textit{second-order cone} (\textit{quadratic cone}, \textit{Lorentz cone}, or \textit{ice-cream cone}) in $\mathbb{R}^n$ is the set
\[ \mathcal{Q}_n \equiv \left\{  x = \left(x_0; \bar{x} \right) \in \mathbb{R}\times \mathbb{R}^{n-1} \colon x_0 \geq \lVert \bar{x} \rVert   \right\} .
\]
The \textit{rotated quadratic cone} is obtained by rotating the second-order cone by 45 degrees in the $x_0$-$x_1$ plane:
\[ 
\hat{\mathcal{Q}}_n \equiv \left\{  x = \left(x_0; x_1; \hat{x} \right) \in \mathbb{R}\times \mathbb{R} \times \mathbb{R}^{n-2}  \colon   2x_0 x_1 \geq \lVert \hat{x} \rVert^2 , \,  x_0 \geq 0, \,  x_1 \geq 0   \right\} .
\]
The nonnegative orthant is a one-dimensional second-order cone.
Because a second-order cone induces a partial ordering, an $n$-dimensional vector $x \in \mathcal{Q}_n$ can be represented as  $x \succeq_{\mathcal{Q}_n} 0$.  The subscript $n$ is sometimes omitted when it is clear from the context.

Second-order cone programming is an extension of linear programming.  In second-order cone programming, one minimizes a linear objective function under linear equality constraints and second-order cone constraints where variables are required to be in the second-order cones.   Let $x_i \, (i=1, \dots, r)$ be vectors not necessarily of the same dimensions.  Let $c_i \, (i=1, \dots, r)$ and $b$ be vectors and $A_i \, (i=1, \dots, r)$ be matrices.
 The standard form second-order cone programming problem is 
 \[ 
 \begin{array}{ll}
 \min_x  & \sum_{i=1}^r c_i^\top x_i \\
\text{subject to} & \sum_{i=1}^r A_i x_i = b\\
& x_i \succeq_{Q} 0
 \end{array}
 \]

Second-order cone programming has many applications.  A solution to a second-order cone programming problem can be obtained approximately by interior point methods in polynomial time of the problem data size.  See~\cite{MR1971381} for a survey of applications and algorithms of second-order cone programming.
In addition, the complexity of an interior point method for second-order cone programming doesn't depend on the dimension of the second-order cone.

The metamodel fitting problem~\eqref{prob:full} can be casted as a second-order cone programming problem.  Below are two constructions.
\paragraph{Model I}

\[
\begin{gathered}
\min_{\alpha, z} \qquad  z \\
\text{subject to} 
  \sum_{j=i-3}^{i} a^{(i)}_{3j} \alpha_j   = c_{11}^{(i)} -  d_{11}^{(i)},\\ 
  \sum_{j=i-3}^{i} a^{(i)}_{2j} \alpha_j =  -t_i c_{11}^{(i)} + 2 c_{12}^{(i)} + t_{i+1} d_{11}^{(i)} - 2 d_{12}^{(i)}, \\ 
  \sum_{j=i-3}^{i} a^{(i)}_{1j} \alpha_j = - 2 t_i c_{12}^{(i)} + c_{22}^{(i)} + 2 t_{i+1} d_{12}^{(i)} - d_{22}^{(i)}, \\
  \sum_{j=i-3}^{i} a^{(i)}_{0j} \alpha_j = -t_i c_{22}^{(i)} + t_{i+1} d_{22}^{(i)} \\
  \left(c_{11}^{(i)},  c_{22}^{(i)}, \sqrt{2} c_{12}^{(i)} \right)^\top \in \hat{\mathcal{Q}}_2,    \quad  \left(d_{11}^{(i)},  d_{22}^{(i)}, \sqrt{2} d_{12}^{(i)} \right)^\top \in \hat{\mathcal{Q}}_2, \qquad (i=1, \dots, n)  \\
  \left[ z, y_1 - \sum_{j=1}^n \alpha_j B_{j4}(x_1), \dots, y_m - \sum_{j=1}^n \alpha_j B_{j4}(x_m), \sqrt{\lambda} \left( \alpha_3 - 2\alpha_{2} + \alpha_{1} \right), \dots, \sqrt{\lambda} \left( \alpha_n - 2\alpha_{n-1} + \alpha_{n-2} \right) \right]^\top \in \mathcal{Q}_{m+n+1} .
\end{gathered}
\]

\paragraph{Model II.}
The square of the $L_2$ norm of a vector $x \in \mathbb{R}^n$ no more than the value of $y \in \mathbb{R}$ can be represented as a second-order cone constraint:
\[ \lVert x \rVert^2_2 \leq y \;  \Longleftrightarrow \;  \left(y+1; y-1; 2x\right) \in Q_{n+2} . \]
Therefore, problem~\eqref{prob:full} can also be formulated as the following second-order cone programming model:

\[
\begin{gathered}
\min_{\alpha, u, v} \qquad  \sum_{i=1}^m u_i + \lambda \sum_{j=3}^n v_j \\
\text{subject to} 
  \sum_{j=i-3}^{i} a^{(i)}_{3j} \alpha_j   = c_{11}^{(i)} -  d_{11}^{(i)},\\ 
  \sum_{j=i-3}^{i} a^{(i)}_{2j} \alpha_j =  -t_i c_{11}^{(i)} + 2 c_{12}^{(i)} + t_{i+1} d_{11}^{(i)} - 2 d_{12}^{(i)}, \\ 
  \sum_{j=i-3}^{i} a^{(i)}_{1j} \alpha_j = - 2 t_i c_{12}^{(i)} + c_{22}^{(i)} + 2 t_{i+1} d_{12}^{(i)} - d_{22}^{(i)}, \\
  \sum_{j=i-3}^{i} a^{(i)}_{0j} \alpha_j = -t_i c_{22}^{(i)} + t_{i+1} d_{22}^{(i)} \\
  \left(c_{11}^{(i)},  c_{22}^{(i)}, \sqrt{2} c_{12}^{(i)} \right)^\top \in \hat{\mathcal{Q}}_2,    \quad  \left(d_{11}^{(i)},  d_{22}^{(i)}, \sqrt{2} d_{12}^{(i)} \right)^\top \in \hat{\mathcal{Q}}_2, \qquad (i=1, \dots, n)  \\
\left[  u_p +1, u_p - 1, 2 \left( y_ - \sum_{j=1}^n \alpha_j B_{j4}(x_p) \right) \right]^\top \in \mathcal{Q}_3, \qquad (p = 1, \dots, m) \\
\left[ v_q +1, v_q - 1, 2 \left( \alpha_q - 2\alpha_{q-1} + \alpha_{q-2} \right) \right]^\top \in \mathcal{Q}_3, \qquad (q = 3, \dots, n) .
\end{gathered}
\]

\section{Numerical Examples}
We have implemented the nonnegative P-spline regression by second-order cone programming in MATLAB / GNU Octave~\cite{Octave}.

\subsection{Parameter Selection}
We use equally spaced knots.
For each fixed knot sequence, the smoothness parameter $\lambda$ is chosen to minimize the GCV (Generalized Cross-Validation) statistic~\cite{MR1944261}:
Let $\hat{\alpha}(\lambda)$ be the optimal coefficients under parameter $\lambda$.  The average squared residuals using $\lambda$ for the regression is
\[ ASR(\lambda) = m^{-1} \sum_{i=1}^m 
 \left[ y_i - \sum_{j=1}^n \hat{\alpha}(\lambda)_j B_{j4}(x_i) \right]^2 .
\]
Denote $D \in \mathbb{R}^{n \times (n-2)}$ as
\[ D = \begin{bmatrix} 1 & & & & \\ -2 & 1 & & & \\ 1 & -2 & 1 & &  \\ & \ddots & \ddots & \ddots & \\
 & & 1 & -2 &  1 \\ & & & 1 & -2 \\ & & & & 1
\end{bmatrix}  .  \]
Then the penalty term in the regression loss function can be represented as $ x^\top DD^\top x$.
Let $X \in \mathbb{R}^{m \times n}$ denote the design matrix whose $i$th row is
\[ X_i = \begin{bmatrix}  B_{1,4, \mathbf{t}}(x_i) &  B_{2, 4, \mathbf{t}}(x_i) & \dots & B_{n, 4, \mathbf{t}}(x_i)  \end{bmatrix}  . \]

The smoother matrix $S(\lambda)$ is defined as
\[ S(\lambda) = X(X^\top X + \lambda D D^\top)^{-1} X^\top .  \]
Then the generalized cross validation statistic is
\[ GCV(\lambda) = \frac{ASR(\lambda)}{\left[1 - m^{-1} tr\left\{ S(\lambda) \right\} \right]^2} . \]
 The value $tr\left\{ S(\lambda) \right\}$ measures the effective degrees of freedom of the fit.

The number of knots is determined by the Akaike information theoretical criterion (AIC)~\cite{AIC}, i.e, we run the algorithm with different number of knots and choose the number $n$ for the model that has the minimum $AIC$:
\[ AIC = m \left\{ \ln  \left[ \sum_{i=1}^m \left( y_i - \sum_{j=1}^n \alpha_j B_{j4}(x_i) \right)^2/m \right] + 1\right\} + 2n  .   \]

\subsection{Numerical Examples}
The second-order cone programming model is solved via SDPT3-4.0~\cite{MR2894668}
through the YALMIP interface~\cite{YALMIP}.  YALMIP is a modeling language that models the problems into standard second-order cone programs.  SDPT3 and SeDuMi are state of the art software for second-order cone programming.  The reformatted SDPT3 and SeDuMi for GNU Octave by Michael Grant are available at the repositories on GitHub.

Below are some numerical examples with MATLAB.
   Data points are depicted by blue ``*''; the B-spline fitted function is the green curve.  We tested the method with number of internal knots from $4$ to $19$.  For each knot sequence, the $\lambda$ is chosen to minimize $GCV$.  The values of $\lambda$ tested are $1.0e\!-4, 1.0e\!-3, \dotsc,  10^3, 10^4$.  The number of knots is determined by minimizing $AIC$.

\paragraph{Density estimation.}
Figure~\ref{fig:density} shows the output of our method for density estimation.
 
For the Poisson probability density function with mean $20$:
\[f(k) = \frac{20^k e^{-20}}{k!} , \]
based on $GCV$ and $AIC$, the best model has $5$ interior knots and $\lambda = 1.0^{-4}$.

For the Gamma probability density function with $\alpha = 2, \beta = 2$:
\[f(x; \alpha, \beta) = \frac{1}{\beta^\alpha \Gamma(\alpha)}x^{\alpha-1} e^{-\frac{x}{\beta} }, \quad
\Gamma(\alpha) = \int_0^\infty x^{\alpha - 1} e^{-x} dx, \]
the best model has $19$ interior knots and $\lambda = 10^{-2}$.

For the Weibull probability density function with $\alpha = 1$, $\beta = 1.5$: 
\[ f(x; \alpha, \beta) = \begin{cases} \frac{\alpha}{\beta^\alpha}x^{\alpha -1} e^{-(x/\beta)^\alpha}  & x \geq 0 \\ 0 & x<0 , \end{cases}
\]
the best model has $19$ interior knots and $\lambda = 10^{-3}$.

For the Pareto probability density function with $\alpha = 1$, $b=1$:
\[ f(x; \alpha, b) = \begin{cases} (\alpha b^{\alpha})/(x^{\alpha + 1}),  & x \geq b \\ 0 & x < b , \end{cases} \]
the best model has $17$ interior knots and $\lambda = 10^{-4}$.

\begin{figure}[!htb]
{
\centering
\includegraphics[width=0.495\textwidth]{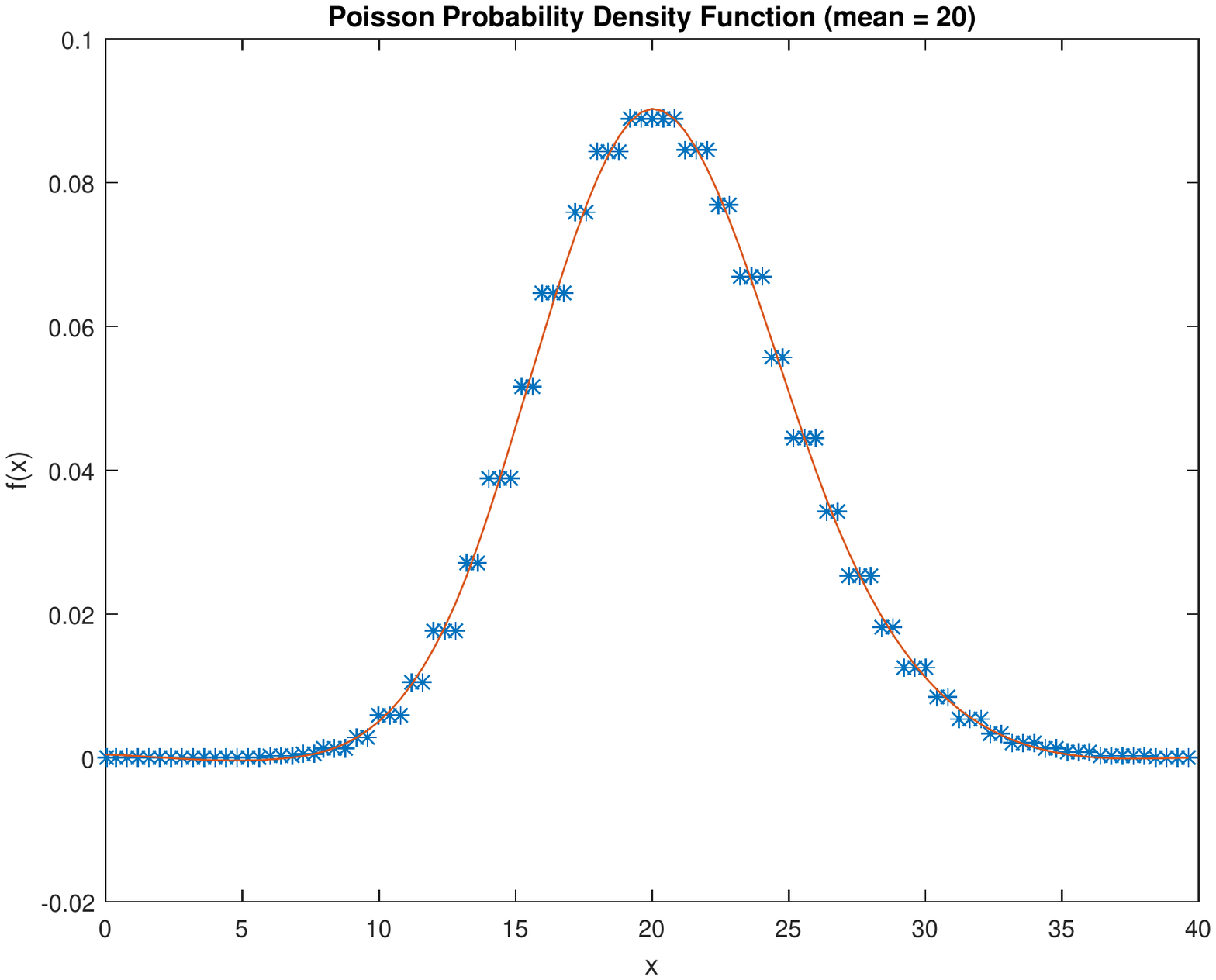}
\includegraphics[width=0.495\textwidth]{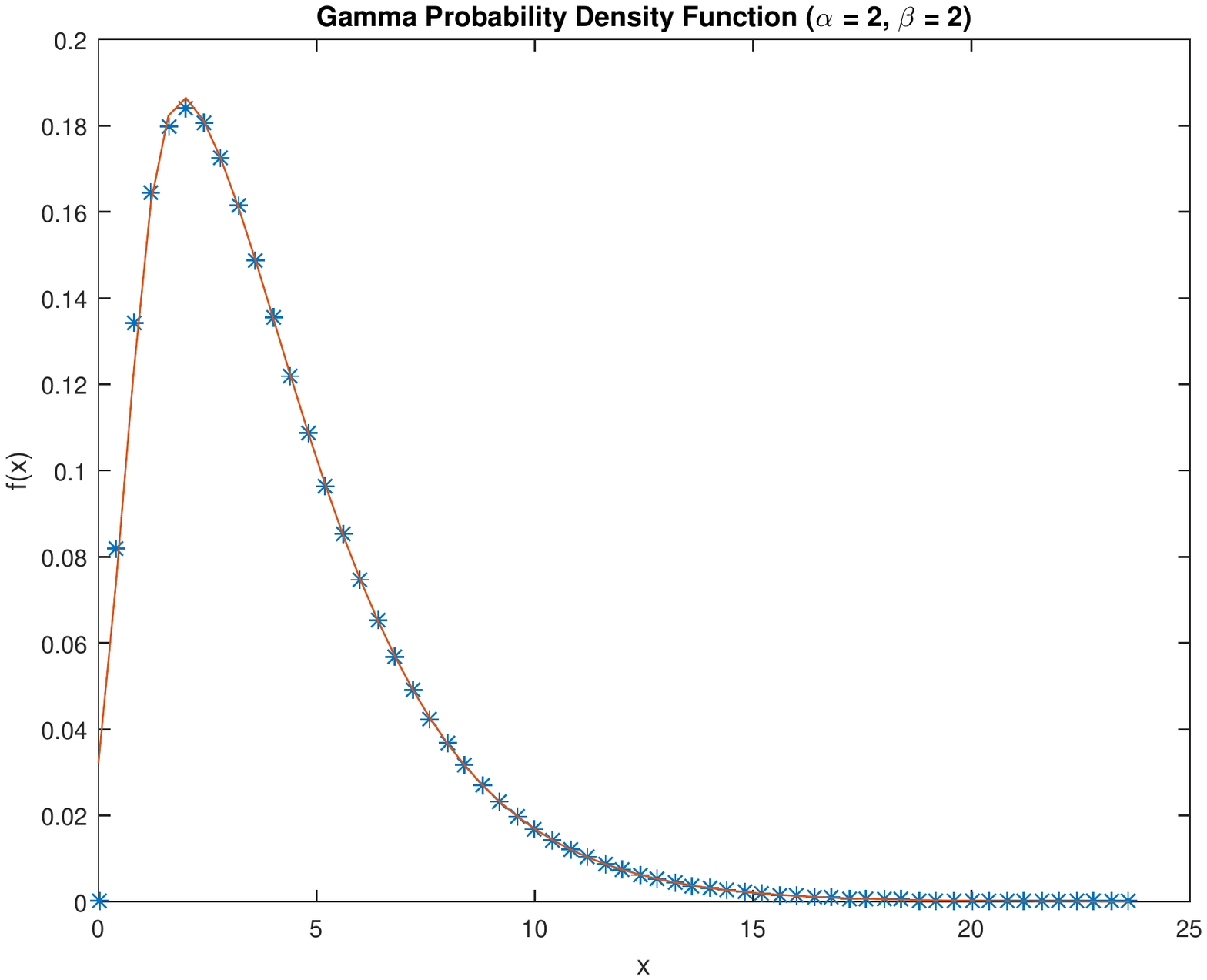}
\includegraphics[width=0.495\textwidth]{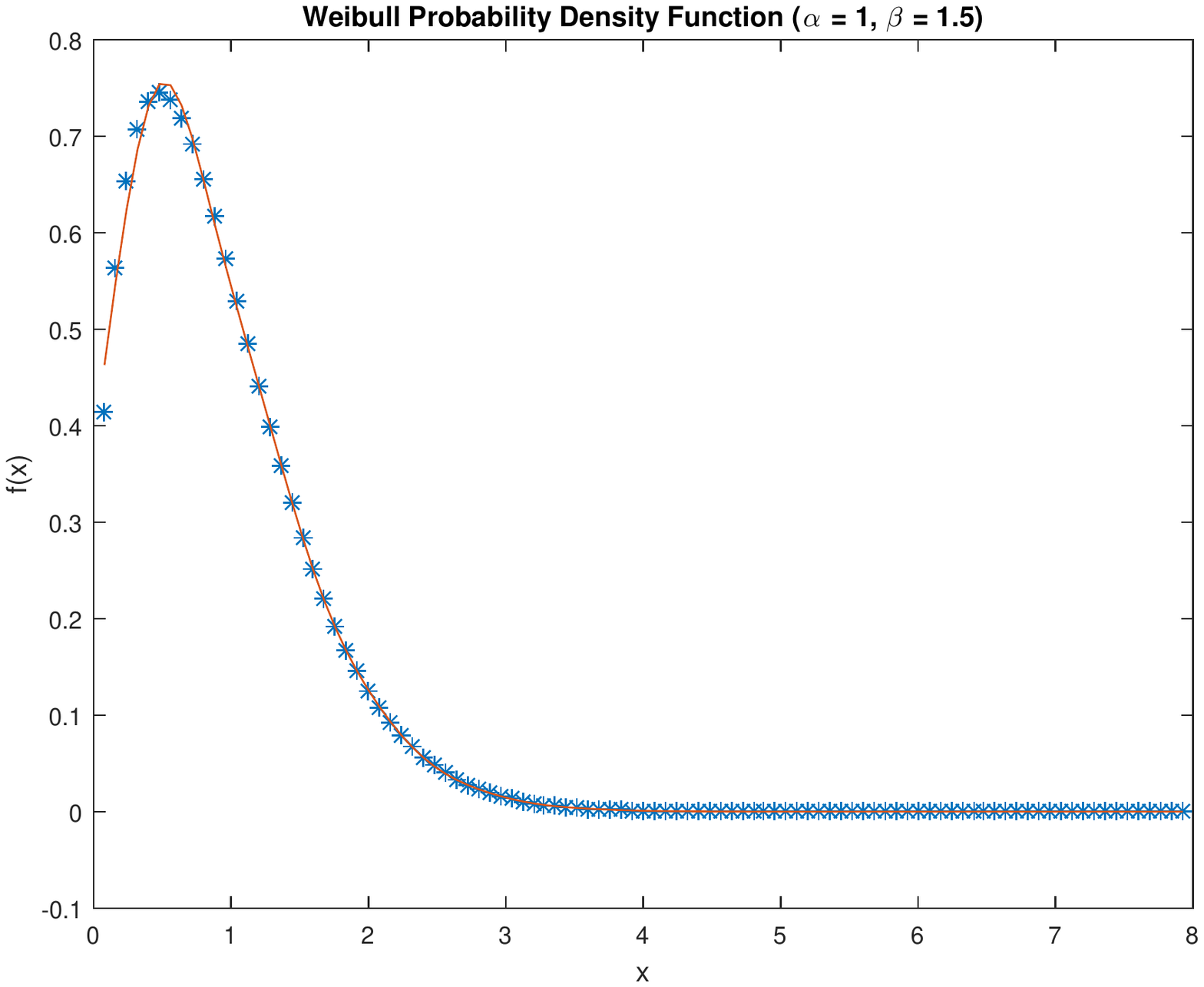}
\includegraphics[width=0.495\textwidth]{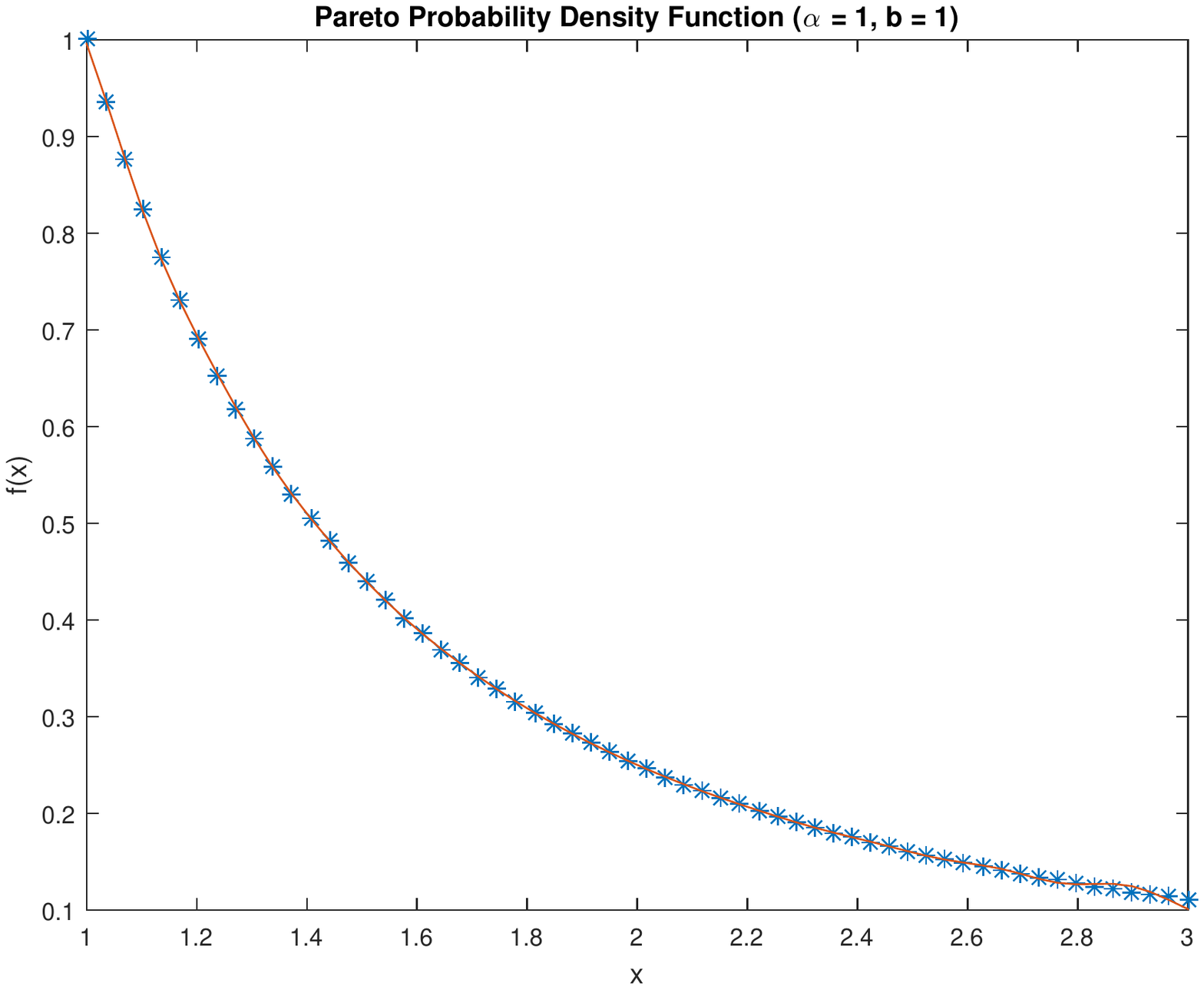}
\caption{Probability Density Estimation} \label{fig:density}
}
\end{figure}
\FloatBarrier

\paragraph{Duration analysis.}
Duration analysis~\cite{MR1167199} studies the spell of an event.
Let $f(t)$ be the density function of the probability distribution of duration.
The survivor function $S(t) = Pr(x \geq t)$ is the probability that the random variable $x$ will equal or exceed $t$.
The hazard function $\lambda(t) = f(t)/S(t)$ is the rate at which spells will be completed at $t$.

The data in Figure~\ref{fig:duration} are strike spells and survivor and hazard estimates from \cite{Kiefer88}.  The duration time are strike lengths in days between 1968 and 1976 involving at least 1,000 works in the U.S. manufacturing industries with major issue.

For the survivor function estimation, the best model has $19$ interior knots and $\lambda = 10^{-3}$;
For the hazard function estimation, the best model has $6$ interior knots and $\lambda = 10^{-2}$;

\begin{figure}[!htb]
{
\centering
\includegraphics[width=0.495\textwidth]{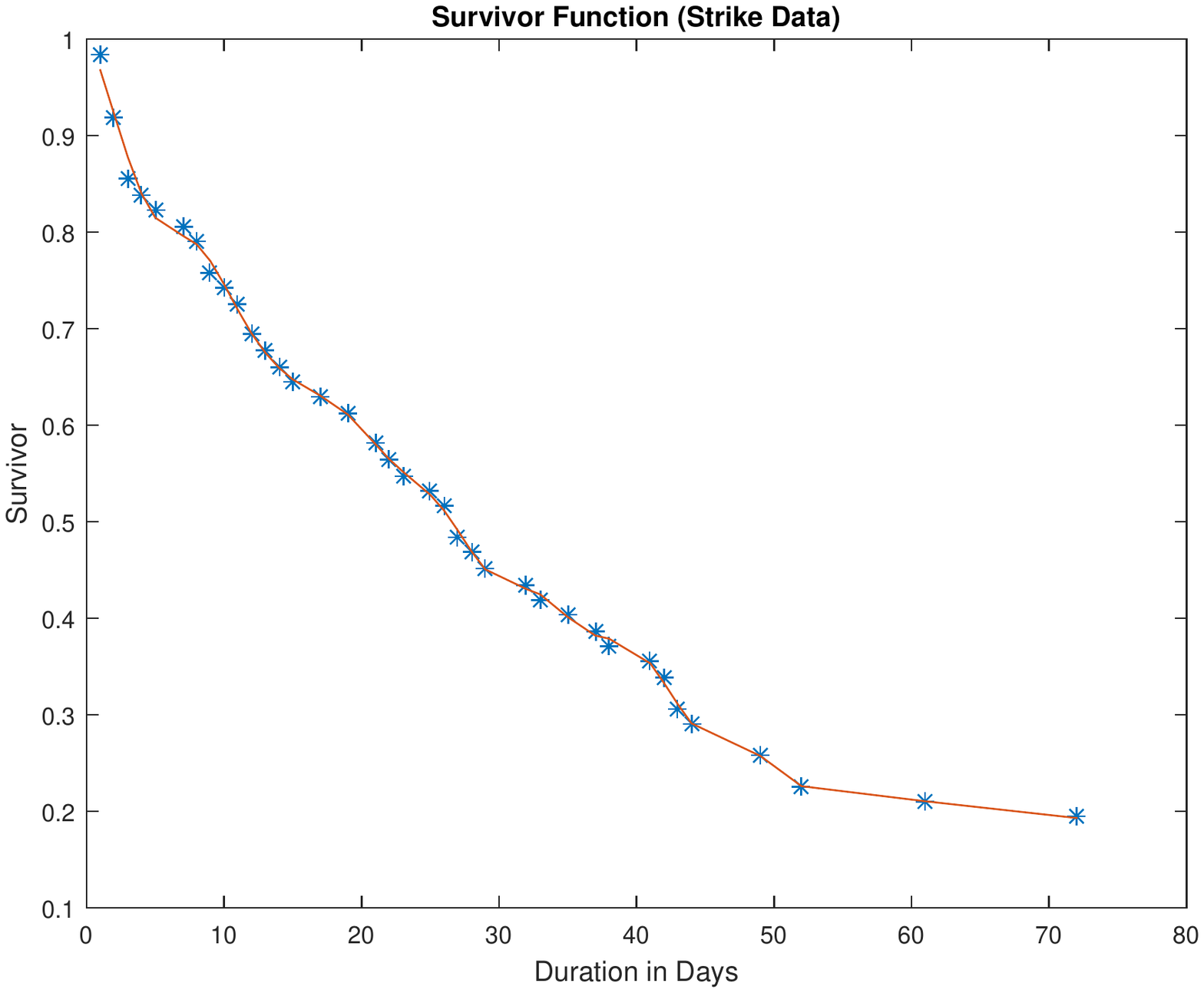}
\includegraphics[width=0.495\textwidth]{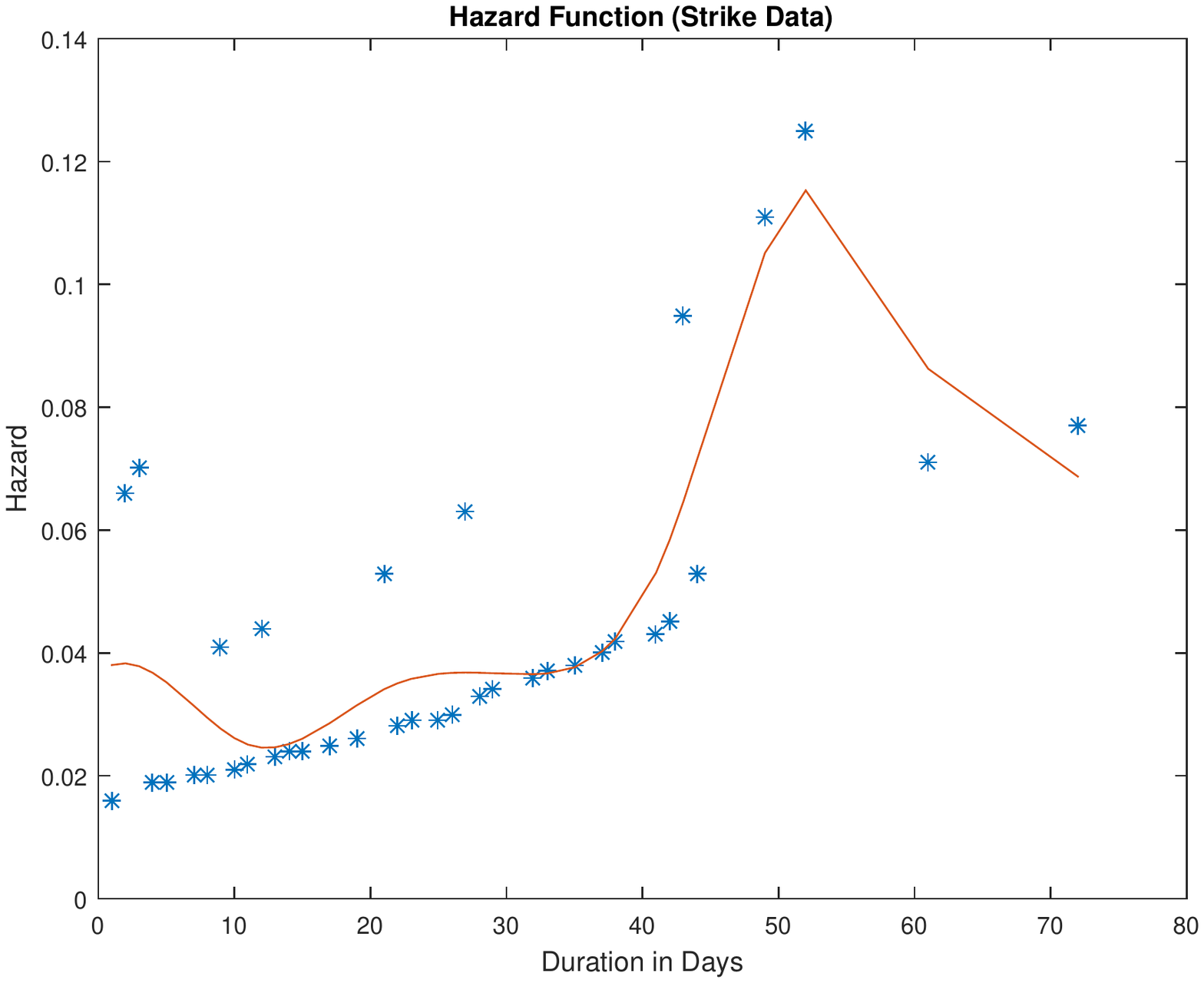}
\caption{Duration Data} \label{fig:duration}
}
\end{figure}

\paragraph{Cost and production.}
The data in Figure~\ref{fig:millcost} are the monthly production costs and output for a hosiery mill over a 4-year period from \cite{millcost}.
The production is in thousands of dozens of pairs, and the costs is in thousands of dollars.
We downloaded the data from Larry Winner's web site: \url{http://www.stat.ufl.edu/~winner/data/millcost.dat}.

For the production data, the best model has $18$ interior knots, and $\lambda = 10^{-2}$.
For the cost data, the best model also has $18$ interior knots, and $\lambda = 10^{-2}$.

\begin{figure}[!htb]
{
\centering
\includegraphics[width=0.495\textwidth]{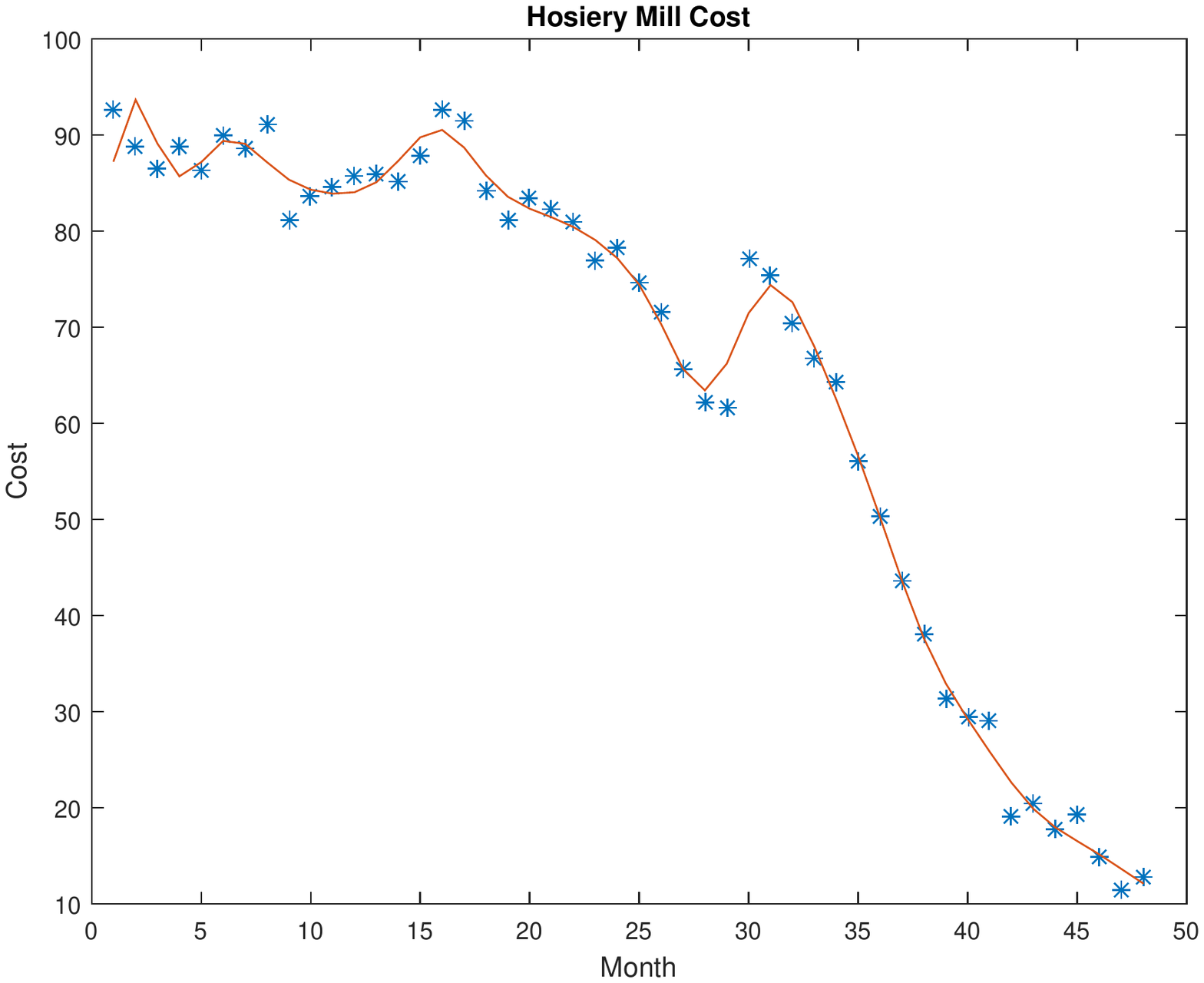}
\includegraphics[width=0.495\textwidth]{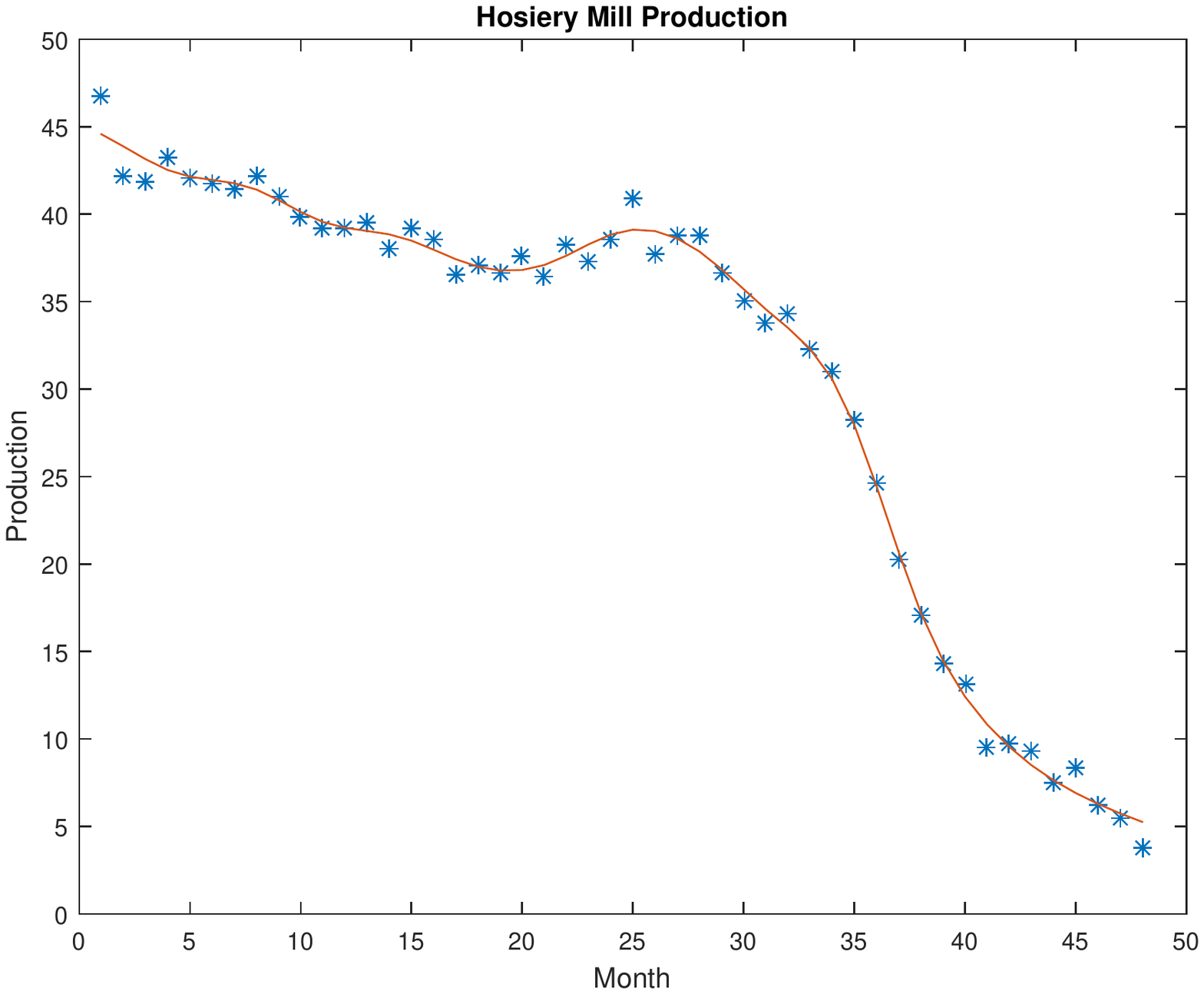}
\caption{Mill Production/Costs} \label{fig:millcost}
}
\end{figure}

\FloatBarrier

\bibliographystyle{plain}
\bibliography{BSpline}

\end{document}